\documentclass[10pt]{amsart}

\usepackage{amsmath,amssymb}
\usepackage[hidelinks]{hyperref}

\newtheorem{theorem}{Theorem}
\newtheorem{corollary}[theorem]{Corollary}
\newtheorem{lemma}[theorem]{Lemma}
\theoremstyle{definition}
\newtheorem{notation}[theorem]{Notation}

\newcommand{\cO}{\mathcal O}
\newcommand{\fm}{\mathfrak m_{\overline K}}
\newcommand{\Kbar}{\overline K}
\newcommand{\ksep}{k^{\mathrm{sep}}}

\title[Integral $2$-adic Tate modules]
{On the integral $2$-adic Tate module of elliptic curves}
\author{Edwina Aylward}
\address{University College London, London WC1H 0AY, UK}
\email{edwina.aylward.23@ucl.ac.uk}

\begin{document}

\begin{abstract}
We show that the integral $2$-adic Tate module of an elliptic curve
over a complete discretely valued field of odd residue characteristic
is determined by its $2$-torsion representation, together with the
square class of $c$ and local information about the pairwise
differences of the roots of $f$ in a model
$E\colon y^2=cf(x)$, where $f$ is monic of degree $3$. The proof uses
explicit halving formulae to determine the Galois action on the full
tower of $2$-power torsion.
\end{abstract}

\maketitle

For an elliptic curve given by a cubic model $E:y^2=cf(x)$, with $f$ monic,
the Galois action on the roots of $f$ determines $E[2]$ as a Galois module.
This information does not in general determine the integral $2$-adic Tate
module $T_2E$. We show that $T_2E$ is determined once one also retains the
square class of $c$ and a little more local information about the relative
positions of the roots of $f(x)$.

Let $K$ be a complete discretely valued field with valuation ring $\cO_K$,
uniformiser $\pi$, and odd residue characteristic $p$. Let $v$ denote the
extension of the valuation of $K$ to $\Kbar$, normalised by $v(\pi)=1$.
Write $\fm=\{x\in\Kbar:v(x)>0\}$, and let $G_K$ denote the absolute Galois
group of $K$. For non-zero $x,y\in\Kbar$, the condition
$x/y\in1+\fm$ says that $x$ and $y$ have the same valuation and the
same first non-zero residue coefficient. Thus condition~(2) below says that
corresponding pairwise differences of roots agree up to an error of strictly
greater valuation.

\begin{theorem}\label{thm:main}
Let $E_1/K:y^2=c_1f_1(x)$ and $E_2/K:y^2=c_2f_2(x)$ be elliptic curves
with $c_1,c_2\in K^\times$ and $f_1,f_2$ monic polynomials of degree $3$
such that:
\begin{enumerate}
\item $c_1/c_2\in K^{\times2}$;
\item there is a $G_K$-equivariant bijection $\phi:\mathcal R_1\to\mathcal R_2$
between the roots of $f_1$ and $f_2$ such that
\[
 \frac{\phi(r)-\phi(s)}{r-s}\in1+\fm
\]
for all distinct $r,s\in\mathcal R_1$.
\end{enumerate}
Then $T_2E_1\cong T_2E_2$ as $\mathbb Z_2[G_K]$-modules.
\end{theorem}

The proof constructs compatible $G_K$-equivariant isomorphisms
$E_1[2^n]\cong E_2[2^n]$ for every $n$. In particular, the two curves have
the same $2$-power torsion fields and the same Galois action on their
$2$-power torsion. The integral representation also determines the
$2$-primary part of the component group of the N\'eron model, including in the case of wild
reduction \cite[Expos\'e IX, \S11]{SGA7}. Hence the $2$-primary parts
of the component groups of $E_1$ and $E_2$ are isomorphic.

The hypotheses of Theorem~\ref{thm:main} are the genus $1$ case of those
appearing in \cite[Theorem~19.1]{DDMM}. For hyperelliptic curves
$C_i:y^2=c_if_i(x)$, the authors prove under these hypotheses that
$V_\ell\operatorname{Jac}C_1\cong V_\ell\operatorname{Jac}C_2$ as
$\mathbb Q_\ell[G_K]$-modules for every $\ell\ne p$. When $K$ is a local field, this also implies equality of the conductor
exponents and local root numbers. For elliptic curves
and $\ell=2$, Theorem~\ref{thm:main} strengthens this rational statement to
an isomorphism of integral Tate modules. It is natural to ask whether the
integral conclusion extends to Jacobians of higher-genus hyperelliptic curves.
A possible approach would be to combine the method used here with Stoll's
algorithm for halving points on odd-degree hyperelliptic Jacobians in Mumford
representation \cite[\S5]{Stoll}.

We also record a variant comparing $K$ with another complete discretely
valued field $K'$ of the same odd residue characteristic $p$, with a fixed
identification of their residue fields with a field $k$. Let $K^{\mathrm t}$
and $(K')^{\mathrm t}$ denote their maximal tamely ramified extensions, and
write
\[
 G_K^{\mathrm t}=\operatorname{Gal}(K^{\mathrm t}/K),\qquad
 G_{K'}^{\mathrm t}=\operatorname{Gal}((K')^{\mathrm t}/K').
\]
Fix a separable closure $\ksep$ of $k$, uniformisers $\pi$ and $\pi'$ of
$K$ and $K'$, and compatible choices of $\pi^{1/n}$ and
$(\pi')^{1/n}$ for every $n$ coprime to $p$. We identify the
prime-to-$p$ roots of unity in $K^{\mathrm t}$ and $(K')^{\mathrm t}$
through their reductions in $\ksep$. These choices determine an isomorphism
$\Psi:G_K^{\mathrm t}\xrightarrow{\sim}G_{K'}^{\mathrm t}$: the actions of
$\sigma$ and $\Psi(\sigma)$ on $\ksep$ agree, and
$\sigma(\pi^{1/n})=\zeta_n^b\pi^{1/n}$ if and only if
$\Psi(\sigma)((\pi')^{1/n})=(\zeta_n')^b(\pi')^{1/n}$.

Let $v_K$ and $v_{K'}$ denote the normalised valuations on the two tame
extensions. For $0\ne x\in K^{\mathrm t}$, write $v_K(x)=m/n$ in lowest
terms and define
\[
 \widetilde{x}_K=\overline{\frac{x}{(\pi^{1/n})^m}}\in(\ksep)^\times,
\]
where $\overline{\ \cdot \ }$ denotes reduction. Define $\widetilde{x'}_{K'}$ similarly. Write $x\sim_{K,K'}x'$ if either
$x=x'=0$, or $v_K(x)=v_{K'}(x')$ and
$\widetilde{x}_K=\widetilde{x'}_{K'}$. For $c\in K^\times$ and
$c'\in(K')^\times$, write $c\sim_{\mathrm{sq}}c'$ if
$v_K(c)\equiv v_{K'}(c')\pmod2$ and
$\widetilde c_K/\widetilde{c'}_{K'}\in k^{\times2}$. Equivalently, the
associated quadratic characters correspond under $\Psi$.

\begin{theorem}\label{thm:different}
Let $E/K:y^2=f(x)$ and $E'/K':y^2=f'(x)$ be elliptic curves,
where $f$ and $f'$ are separable cubic polynomials, and let
$\mathcal R$ and $\mathcal R'$ be their sets of roots. Assume that $\mathcal R\subset K^{\mathrm t}$,
$\mathcal R'\subset(K')^{\mathrm t}$. 
Suppose that
\begin{enumerate}
\item the leading coefficients
$c_f,c_{f'}$ satisfy $c_f\sim_{\mathrm{sq}}c_{f'}$,
\item there is a
$\Psi$-equivariant bijection $\phi:\mathcal R\to\mathcal R'$ such that
$r-s\sim_{K,K'}\phi(r)-\phi(s)$ for all distinct $r,s\in\mathcal R$.
\end{enumerate}
Then the actions of $G_K$ and $G_{K'}$ on $T_2E$ and $T_2E'$
factor through $G_K^{\mathrm t}$ and $G_{K'}^{\mathrm t}$, respectively, and
$
T_2E\cong\Psi^*T_2E'
$
as $\mathbb Z_2[G_K^{\mathrm t}]$-modules.
\end{theorem}

Here $\Psi^*T_2E'$ denotes the same $\mathbb Z_2$-module $T_2E'$, with
$G_K^{\mathrm t}$-action $\sigma\cdot P=\Psi(\sigma)P$. Note that the assumptions
$\mathcal R\subset K^{\mathrm t}$ and
$\mathcal R'\subset(K')^{\mathrm t}$ are automatic unless $p=3$.

As an application, Theorem~\ref{thm:main} gives the following explicit local
constancy result.

\begin{corollary}\label{cor:explicit}
Let $E_1/K:y^2=c_1f_1(x)$ and $E_2/K:y^2=c_2f_2(x)$ be elliptic curves
with $c_1,c_2\in K^\times$ and $f_1,f_2\in\cO_K[x]$ monic polynomials of
degree $3$. Let $\mathcal R_1$ denote the set of roots of $f_1$, and put
$d=\max_{r\ne s\in\mathcal R_1}v(r-s)$ and
$e=\max_{r\in\mathcal R_1}v(f_1'(r))$. Suppose that
$c_1/c_2\in K^{\times2}$ and $f_1\equiv f_2\pmod{\pi^N\cO_K[x]}$ for
some integer $N>d+e$. Then $T_2E_1\cong T_2E_2$ as
$\mathbb Z_2[G_K]$-modules.
\end{corollary}

Kisin's work \cite{Kisin} on local constancy of $\ell$-adic Galois
representations implies a non-explicit version of Corollary~\ref{cor:explicit}.
More precisely, for each fixed elliptic curve $E_1$ and each prime
$\ell\ne p$, there is an integer $N_0$ such that
$T_\ell E_1\cong T_\ell E_2$ whenever
$f_1\equiv f_2\pmod{\pi^{N_0}\cO_K[x]}$ and
$c_1/c_2\in K^{\times2}$. Corollary~\ref{cor:explicit} gives an explicit
value when $\ell=2$.

\section{Halving points}

Throughout this section, $K$ denotes the complete discretely valued field
fixed in the introduction. Let $E/K$ be an elliptic curve given over
$\Kbar$ by
$E:y^2=(x-\alpha)(x-\beta)(x-\gamma)$. To lift the given bijection on
$2$-torsion in Theorem~\ref{thm:main} to compatible Galois-equivariant
bijections on $2^n$-torsion, we need an explicit description of half-points
and of how their parameters behave under negation, translation by
$2$-torsion, and addition.

For $P=(a,b)\in E(\Kbar)$, Bekker and Zarhin \cite[\S2]{BekkerZarhin} give
a bijection between the points $Q\in E(\Kbar)$ satisfying $2Q=P$ and the
triples $(t_\alpha,t_\beta,t_\gamma)\in\Kbar^3$ satisfying $t_r^2=a-r$
for $r\in\{\alpha,\beta,\gamma\}$ and
$t_\alpha t_\beta t_\gamma=-b$. We record the identities used below.

\begin{lemma}[\cite{BekkerZarhin}]\label{lem:halving}
If $Q$ corresponds to $(t_\alpha,t_\beta,t_\gamma)$, then:
\begin{enumerate}
\item
$
 x(Q)=a+t_\alpha t_\beta+t_\alpha t_\gamma+t_\beta t_\gamma$, and 
 $y(Q)=-(t_\alpha+t_\beta)(t_\alpha+t_\gamma)(t_\beta+t_\gamma).
$
In particular, $x(Q)-r=(t_r+t_s)(t_r+t_u)$ for every permutation
$r,s,u$ of $\alpha,\beta,\gamma$.
\item The point $-Q$, viewed as a half-point of $-P$, corresponds to
$(-t_\alpha,-t_\beta,-t_\gamma)$.
\item If $T_\alpha=(\alpha,0)$, then $Q+T_\alpha$ corresponds to
$(t_\alpha,-t_\beta,-t_\gamma)$, and similarly for the other two roots.
\end{enumerate}
\end{lemma}

To show that the maps constructed below are group homomorphisms, we also need
to understand the halving triple associated to a sum. The following formula is
not stated in \cite{BekkerZarhin}, but follows by a direct computation from the
usual chord-and-tangent formulae.

\begin{lemma}\label{lem:addition}
Let $A=(x_A,y_A),B=(x_B,y_B)\in E(\Kbar)$, with $A\ne\pm B$. Suppose
that $P$ and $Q$ are half-points of $A$ and $B$, corresponding to
triples $(a_\alpha,a_\beta,a_\gamma)$ and
$(b_\alpha,b_\beta,b_\gamma)$, respectively. Then $P+Q$, viewed as a
half-point of $A+B$, corresponds to $(c_\alpha,c_\beta,c_\gamma)$, where
\[
 c_\alpha=\frac{b_\alpha a_\beta a_\gamma-a_\alpha b_\beta b_\gamma}{x_A-x_B},
 \quad
 c_\beta=\frac{a_\alpha b_\beta a_\gamma-b_\alpha a_\beta b_\gamma}{x_A-x_B},
 \quad
 c_\gamma=\frac{a_\alpha a_\beta b_\gamma-b_\alpha b_\beta a_\gamma}{x_A-x_B}.
\]
In particular, for every permutation $i,j,k$ of $\alpha,\beta,\gamma$,
\[
 c_i+c_j=\frac{a_i b_j+b_i a_j}{a_k+b_k}.
\]
\end{lemma}

We will use the following notation in the remainder of the paper.

\begin{notation}\label{not:sim}
For $a,b\in\Kbar$, write $a\sim b$ if either $a=b=0$, or $a,b\ne0$
and $a/b\in1+\fm$. Thus, if $a,b\ne0$, then $a\sim b$ if and only if
$v(a-b)>v(a)=v(b)$.
\end{notation}

\begin{lemma}\label{lem:sums}
Let $a,a',b,b'\in\Kbar$. Suppose that $a'\sim a$, $b'\sim b$, and
$a'^2-b'^2\sim a^2-b^2$, where $a\ne\pm b$. Then
$a'+b'\sim a+b$ and $a'-b'\sim a-b$.
\end{lemma}

\begin{proof}
If $a=0$ or $b=0$, the result is immediate. Thus we may assume that
$a,a',b,b'$ are all non-zero. For one choice of sign, say
$\varepsilon\in\{\pm1\}$, one has
$v(a+\varepsilon b)=\min(v(a),v(b))$. It follows from $a'\sim a$ and
$b'\sim b$ that $a'+\varepsilon b'\sim a+\varepsilon b$. The other
congruence now follows from
\[
 \frac{(a'+b')(a'-b')}{(a+b)(a-b)}
 =\frac{a'^2-b'^2}{a^2-b^2}\equiv1\pmod{\fm}.
\]
\end{proof}

\begin{lemma}\label{lem:unique}
Let $E:y^2=(x-r_1)(x-r_2)(x-r_3)$ be an elliptic curve over $K$, with
$r_i\in\Kbar$. Suppose that $P=(a,b)$ and $P'=(a',b')$ are non-zero
points of $E[\ell^n]$, where $\ell\ne p$. If $a-r_i\sim a'-r_i$ for
$i=1,2,3$ and $b\sim b'$, then $P=P'$.
\end{lemma}

\begin{proof}
We may pass to a finite extension $L/K$ containing $P,P'$ and
$r_1,r_2,r_3$, and over which $E$ has semistable reduction. Choose a
minimal Weierstrass equation over $L$. Since the residue characteristic is
odd, after completing the square we may write it as
$E_{\min}:Y^2=(X-\widetilde r_1)(X-\widetilde r_2)(X-\widetilde r_3)$. The
two equations are related by the change of variables $x=u^2X+q$, $y=u^3Y$
for some $u\in L^\times$ and $q\in L$. Write $P=(A,B)$,
$P'=(A',B')$ for the coordinates of the same points on $E_{\min}$. Then
$a-r_i=u^2(A-\widetilde r_i)$, $a'-r_i=u^2(A'-\widetilde r_i)$, and so
$A-\widetilde r_i\sim A'-\widetilde r_i$ for $i=1,2,3$. Similarly,
$B\sim B'$. It follows that we may assume that the given equation is minimal
and has semistable reduction.

The formal group of a minimal Weierstrass equation has no non-trivial
$\ell$-power torsion, since $\ell\ne p$
\cite[Proposition~VII.3.1(a)]{Silverman}. Thus $P$ and $P'$ have integral
coordinates. The roots $r_1,r_2,r_3$ are also integral.

If $a=a'$, then $b'=\pm b$, and $b\sim b'$ implies $b=b'$ since $p$
is odd. If $a=r_i$ for some $i$, then $a-r_i\sim a'-r_i$ forces
$a'=r_i$, and again $P=P'$. We may therefore assume that $a\ne a'$ and
$a\ne r_i$ for every $i$.

Put $w_i=v(a-r_i)$ for $i=1,2,3$, and put $h=v(a-a')$. The hypotheses
give $h>w_i$ for every $i$, and $w_i\ge0$. Since the equation for $E$
is minimal and semistable, there exists $j,k\in\{1,2,3\}$ such that
$v(r_j-r_k)=0$. Indeed, the reduction is either smooth or nodal, so the three
roots cannot all have the same reduction. Then
$0=v(r_j-r_k)\ge\min(w_j,w_k)$, and hence $\min_iw_i=0$. Let $\lambda$
be the slope of the line through $P$ and $-P'$. Since $b\sim b'$ and
$p$ is odd, $v(b+b')=v(b)=\tfrac12(w_1+w_2+w_3)$. Thus
$v(\lambda)=\tfrac12(w_1+w_2+w_3)-h<0$, as one of the $w_i$ is zero and
the other two are strictly less than $h$.

Since $x(P-P')=\lambda^2+r_1+r_2+r_3-a-a'$ and all terms other than
$\lambda^2$ are integral, we have $v(x(P-P'))=2v(\lambda)<0$. It follows
that $P-P'$ belongs to the formal group of the minimal equation. Since it is
an $\ell$-power torsion point and $\ell\ne p$, it must be trivial. Therefore
$P=P'$.
\end{proof}

\begin{lemma}\label{lem:xdifference}
Let $E_1:y^2=\prod_{r\in\mathcal R_1}(x-r)$ and
$E_2:y^2=\prod_{r'\in\mathcal R_2}(x-r')$ be elliptic curves over $K$,
such that $\phi:\mathcal R_1\xrightarrow{\sim}\mathcal R_2$ is a bijection
with $r'=\phi(r)$ for $r\in\mathcal R_1$. Suppose that
$P,Q\in E_1[\ell^n]$ and $P',Q'\in E_2[\ell^n]$ are non-trivial points
satisfying
\[
 x(P')-r'\sim x(P)-r,\quad y(P')\sim y(P),\quad
 x(Q')-r'\sim x(Q)-r,\quad y(Q')\sim y(Q)
\]
for every $r\in\mathcal R_1$. If $P\ne\pm Q$, then
$x(P')-x(Q')\sim x(P)-x(Q)$.
\end{lemma}

\begin{proof}
There is some $r\in\mathcal R_1$ such that $x(P)-r\not\sim x(Q)-r$.
Indeed, otherwise multiplying the three congruences gives
$y(P)^2\sim y(Q)^2$. If these coordinates are non-zero, then
$y(P)\sim\varepsilon y(Q)$ for some $\varepsilon\in\{\pm1\}$, and
Lemma~\ref{lem:unique} implies that $P=\varepsilon Q$, a contradiction. The
case where one of the $y$-coordinates is zero is immediate. For this choice
of $r$, there is no cancellation in the leading terms of
$(x(P)-r)-(x(Q)-r)$. Subtracting and using the corresponding primed
congruences gives $x(P')-x(Q')\sim x(P)-x(Q)$.
\end{proof}

\section{Proof of the main results}

\begin{proof}[Proof of Theorem~\ref{thm:main}]
Since $c_1/c_2\in K^{\times2}$, after replacing $E_2$ by a
$K$-isomorphic equation we may assume that $c_1=c_2=c$. The two curves
are then the same quadratic twists of $\widetilde E_i:y^2=f_i(x)$, so it is
enough to prove the result when $c=1$ (as
$T_2E_i\cong T_2\widetilde E_i\otimes\chi_c$, where $\chi_c$ is the
quadratic character corresponding to $K(\sqrt c)/K$).

Write $E_1:y^2=\prod_{r\in\mathcal R_1}(x-r)$ and
$E_2:y^2=\prod_{r\in\mathcal R_1}(x-r')$, where $r'=\phi(r)$. We
inductively construct maps $\varphi_n:E_1[2^n]\to E_2[2^n]$ such that:
\begin{enumerate}
\item $\varphi_n$ restricts to $\varphi_{n-1}$ on $E_1[2^{n-1}]$ and
$2\varphi_n(P)=\varphi_{n-1}(2P)$;
\item for every $P\in E_1[2^n]\setminus\{O\}$ and every
$r\in\mathcal R_1$, $x(\varphi_n(P))-r'\sim x(P)-r$ and
$y(\varphi_n(P))\sim y(P)$, where $\sim$ is defined in
Notation~\ref{not:sim};
\item $\varphi_n$ is a $G_K$-equivariant group isomorphism.
\end{enumerate}

Set $\varphi_0(O)=O$. For $n=1$, define $\varphi_1(O)=O$ and
$\varphi_1((r,0))=(r',0)$. This satisfies the required properties.

Assume that $\varphi_{n-1}$ has been constructed. We set
$\varphi_n(P)=\varphi_{n-1}(P)$ for $P\in E_1[2^{n-1}]$. Let
$P\in E_1[2^n]\setminus E_1[2^{n-1}]$, and let
$(t_r)_{r\in\mathcal R_1}$ be the triple corresponding to $P$ as a
half-point of $2P$. For each $r\in\mathcal R_1$, choose $t_r'$ as
follows. If $t_r=0$, set $t_r'=0$. Otherwise, let $t_r'$ be the unique
square root of $x(\varphi_{n-1}(2P))-r'$ such that
$t_r'/t_r\in1+\fm$. Such a choice exists and is unique because $p\ne2$
and
\[
 \frac{x(\varphi_{n-1}(2P))-r'}{t_r^2}\in1+\fm.
\]

The triple $(t_r')_{r\in\mathcal R_1}$ satisfies the required product
condition of Lemma~\ref{lem:halving}. Indeed, this is immediate if some
$t_r=0$. Otherwise,
$\varepsilon=-(\prod_{r\in\mathcal R_1}t_r')/y(\varphi_{n-1}(2P))$
satisfies $\varepsilon^2=1$ and $\varepsilon\in1+\fm$. Since $p\ne2$,
we have $\varepsilon=1$. We define $\varphi_n(P)$ to be the half-point of
$\varphi_{n-1}(2P)$ corresponding to $(t_r')_{r\in\mathcal R_1}$.

Let $r,s\in\mathcal R_1$ be distinct. We have $t_r'\sim t_r$,
$t_s'\sim t_s$, and
$(t_r')^2-(t_s')^2=s'-r'\sim s-r=t_r^2-t_s^2$. Therefore
Lemma~\ref{lem:sums} gives $t_r'+t_s'\sim t_r+t_s$ and
$t_r'-t_s'\sim t_r-t_s$. It follows from the coordinate formulae in
Lemma~\ref{lem:halving} that
\[
 x(\varphi_n(P))-r'=(t_r'+t_s')(t_r'+t_u')
 \sim(t_r+t_s)(t_r+t_u)=x(P)-r
\]
for every permutation $r,s,u$ of the three roots, and
\[
 y(\varphi_n(P))=-\prod_{\{r,s\}\subset\mathcal R_1}(t_r'+t_s')
 \sim-\prod_{\{r,s\}\subset\mathcal R_1}(t_r+t_s)=y(P).
\]
Thus property~(2) holds.

We next prove $G_K$-equivariance. Let $\sigma\in G_K$. If $t_r'$ is the
square root chosen above, then
$\sigma(t_r')^2=x(\varphi_{n-1}(2\sigma(P)))-\phi(\sigma(r))$ and
$\sigma(t_r')/\sigma(t_r)\in1+\fm$. Hence $\sigma(t_r')$ is precisely the
square root attached to $\phi(\sigma(r))$ in the construction of
$\varphi_n(\sigma(P))$. Applying $\sigma$ to the formulae of
Lemma~\ref{lem:halving} gives $\varphi_n(\sigma(P))=\sigma(\varphi_n(P))$.

It remains to prove that $\varphi_n$ is a group isomorphism. The sign-change
observations recorded in Lemma~\ref{lem:halving}, together with the inductive
hypothesis, give
\[
 \varphi_n(-P)=-\varphi_n(P),\qquad
 \varphi_n(P+T)=\varphi_n(P)+\varphi_1(T) \tag{$\dagger$}
\]
for every $P\in E_1[2^n]$ and $T\in E_1[2]$. These identities imply that
$\varphi_n$ is injective. Indeed, if $\varphi_n(P)=\varphi_n(Q)$, then the
inductive hypothesis gives $2P=2Q$, so $Q=P+T$ for some $T\in E_1[2]$.
Hence $\varphi_1(T)=O$, and therefore $T=O$. Since $E_1[2^n]$ and
$E_2[2^n]$ have the same order, $\varphi_n$ is bijective.

We now prove additivity. Let $P,Q\in E_1[2^n]$. Additivity is immediate from
the second identity in $(\dagger)$ if one of $P,Q$ belongs to $E_1[2]$.
The same identities also give additivity when $2P=\pm2Q$. For example, if
$2P=2Q$, then $Q=P+T$ for some $T\in E_1[2]$, and
\[
 \varphi_n(P)+\varphi_n(Q)=2\varphi_n(P)+\varphi_1(T)
 =\varphi_{n-1}(2P)+\varphi_1(T)=\varphi_n(2P+T)=\varphi_n(P+Q).
\]
The case $2P=-2Q$ is similar, using $\varphi_n(-P)=-\varphi_n(P)$.

We may therefore assume that $2P\ne\pm2Q$. Write
$\mathcal R_1=\{r_1,r_2,r_3\}$ and put $r_i'=\phi(r_i)$. Let $(a_i)$,
$(b_i)$, and $(c_i)$ be the triples corresponding to $P$, $Q$, and
$P+Q$ as half-points of $2P$, $2Q$, and $2P+2Q$, respectively. Let
$(a_i')$, $(b_i')$, and $(c_i')$ be the corresponding triples defining
$\varphi_n(P)$, $\varphi_n(Q)$, and $\varphi_n(P+Q)$. By construction,
$a_i'\sim a_i$, $b_i'\sim b_i$, and $c_i'\sim c_i$ for each $i$.
Finally, let $(d_i')$ be the triple corresponding to
$\varphi_n(P)+\varphi_n(Q)$ as a half-point of
$\varphi_{n-1}(2P)+\varphi_{n-1}(2Q)=\varphi_{n-1}(2P+2Q)$. By
Lemma~\ref{lem:xdifference},
\[
 (a_i')^2-(b_i')^2=x(\varphi_{n-1}(2P))-x(\varphi_{n-1}(2Q))
 \sim x(2P)-x(2Q)=a_i^2-b_i^2.
\]
As $a_i'\sim a_i$ and $b_i'\sim b_i$, it follows from
Lemma~\ref{lem:sums} that $a_i'\pm b_i'\sim a_i\pm b_i$. For distinct
$i,j$,
\[
 (a_i b_j-b_i a_j)(a_i b_j+b_i a_j)=(r_i-r_j)(x(2P)-x(2Q)),
\]
and the analogous identity holds for the primed terms. Using
$r_i'-r_j'\sim r_i-r_j$ and Lemma~\ref{lem:sums}, it follows that
$a_i'b_j'\pm b_i'a_j'\sim a_i b_j\pm b_i a_j$.

Let $i,j,k$ be a permutation of $1,2,3$. By Lemma~\ref{lem:addition},
\[
 c_i+c_j=\frac{a_i b_j+b_i a_j}{a_k+b_k},\qquad
 d_i'+d_j'=\frac{a_i'b_j'+b_i'a_j'}{a_k'+b_k'}.
\]
Therefore $d_i'+d_j'\sim c_i+c_j$. On the other hand,
$c_i'\sim c_i$, $c_j'\sim c_j$, and
$(c_i')^2-(c_j')^2=r_j'-r_i'\sim r_j-r_i=c_i^2-c_j^2$, so
Lemma~\ref{lem:sums} gives $c_i'+c_j'\sim c_i+c_j$. Hence
$c_i'+c_j'\sim d_i'+d_j'$ for all distinct $i,j$.

Set $A=\varphi_n(P+Q)$ and $B=\varphi_n(P)+\varphi_n(Q)$. It follows from
the formulae in Lemma~\ref{lem:halving} that
$x(A)-r_i'\sim x(B)-r_i'$ for $i=1,2,3$, and $y(A)\sim y(B)$.
Lemma~\ref{lem:unique} therefore implies that $A=B$. Thus $\varphi_n$ is
additive, and hence is a $G_K$-equivariant group isomorphism. This completes
the induction.

The maps $\varphi_n$ are compatible with multiplication by $2$. Passing to
the inverse limit therefore yields an isomorphism $T_2E_1\cong T_2E_2$ of
$\mathbb Z_2[G_K]$-modules.
\end{proof}

\begin{proof}[Proof of Theorem~\ref{thm:different}]
Since the roots of $f$ lie in $K^{\mathrm t}$, wild inertia acts trivially on
$E[2]$. Its image on $T_2E$ therefore lies in the pro-$2$ group
$\ker(\operatorname{GL}_2(\mathbb Z_2)\to\operatorname{GL}_2(\mathbb F_2))$.
As wild inertia is pro-$p$ and $p\ne2$, it acts trivially on $T_2E$. Thus
the action factors through $G_K^{\mathrm t}$, and similarly for $E'$.

The relation $\sim_{K,K'}$ has the same product, quotient, and square-root
properties as $\sim$, so the arguments of Lemmas~\ref{lem:sums} and
\ref{lem:xdifference} apply unchanged. Also, if two elements of $(K')^{\mathrm t}$ are
$\sim_{K,K'}$-equivalent to the same element of $K^{\mathrm t}$, then
their quotient lies in $1+\mathfrak m_{(K')^{\mathrm t}}$. Thus
Lemma~\ref{lem:unique} may be applied on $E'$.

The condition $c_f\sim_{\mathrm{sq}}c_{f'}$ allows us to twist separately
and reduce to the case where $f$ and $f'$ are monic. The proof of
Theorem~\ref{thm:main} then applies with $\sim$ replaced by
$\sim_{K,K'}$, giving compatible $\Psi$-equivariant isomorphisms
$E[2^n]\cong E'[2^n]$ for every $n$, and hence
$T_2E\cong\Psi^*T_2E'$.
\end{proof}

\begin{proof}[Proof of Corollary~\ref{cor:explicit}]
Write $\mathcal R_1=\{r_1,r_2,r_3\}$. Fix $r_1$, and assume that
$a=v(r_1-r_2)\le b=v(r_1-r_3)$. Then $b\le d$ and
$a+b=v(f_1'(r_1))\le e$, so $N>a+2b$. Since
$f_1\equiv f_2\pmod{\pi^N\cO_K[x]}$ and $r_1$ is integral, we have
$v(f_2(r_1))\ge N$ and $v(f_2'(r_1)-f_1'(r_1))\ge N$. Hence
$v(f_2'(r_1))=a+b$. Moreover,
$\tfrac12f_1''(r_1)=(r_1-r_2)+(r_1-r_3)$, so
$v(\tfrac12f_2''(r_1))\ge a$.

Now
\[
 f_2(r_1+x)=f_2(r_1)+f_2'(r_1)x+\tfrac12f_2''(r_1)x^2+x^3.
\]
If $f_2(r_1)\ne0$, put $c=v(f_2(r_1))$. Since $c\ge N>a+2b$, the
Newton polygon of $f_2(r_1+x)$ has a unique segment of slope less than
$-b$, namely the segment from $(0,c)$ to $(1,a+b)$. Hence $f_2$ has
a unique root $\phi(r_1)$ satisfying
$v(\phi(r_1)-r_1)=c-a-b>b$. If $f_2(r_1)=0$, take
$\phi(r_1)=r_1$; the same coefficient estimates show that this is again
the unique root satisfying $v(\phi(r_1)-r_1)>b$.

Applying the same argument to each $r_i$ gives a $G_K$-equivariant
bijection $\phi:\mathcal R_1\to\mathcal R_2$ such that
$v(\phi(r_i)-r_i)>v(r_i-r_j)$ for $i\ne j$. Finally, for $i\ne j$,
\[
 v\bigl((\phi(r_i)-\phi(r_j))-(r_i-r_j)\bigr)>v(r_i-r_j),
\]
and therefore
$(\phi(r_i)-\phi(r_j))/(r_i-r_j)\in1+\fm$. The result follows from
Theorem~\ref{thm:main}.
\end{proof}

\section*{Acknowledgements}

I thank Vladimir Dokchitser for suggesting this problem, and in particular for
directing me to Zarhin's work on division by $2$ on hyperelliptic curves. This
work was supported by the Engineering and Physical Sciences Research Council
[EP/S021590/1], the EPSRC Centre for Doctoral Training in Geometry and Number
Theory (The London School of Geometry and Number Theory) at University College
London.

\end{document}